\documentclass[12pt]{amsart}
\usepackage[margin=.8in]{geometry}
\usepackage[scr=rsfs]{mathalpha}
\usepackage{xspace}
\usepackage{amsfonts}
\usepackage{amsthm}
\usepackage{amssymb}
\usepackage{times} 
\usepackage{graphicx}
\usepackage{hyperref}
\usepackage{tikz}
\usepackage{comment}
\usetikzlibrary{arrows}
\usepackage[all]{xy}
\usepackage{tikz-cd}
\usepackage{enumerate}

\usepackage{mathtools}

\date{}

\numberwithin{equation}{section}
\newtheorem{thm}{Theorem}[section]

\tikzset{node distance=5cm, auto}

\newtheorem{lemma}[thm]{Lemma}
\newtheorem{conjecture}[thm]{Conjecture}

\newtheorem{remark}[thm]{Remark}

\makeatletter
\newcommand{\etale}{\'etal\@ifstar{\'e}{e\xspace}}
\makeatother
\usepackage[OT2,T1]{fontenc}
\DeclareSymbolFont{cyrletters}{OT2}{wncyr}{m}{n}
\DeclareMathSymbol{\Sha}{\mathalpha}{cyrletters}{"58}

\begin{document}
\title{Totally ramified subfields of $p$-algebras  over  discrete valued fields with imperfect residue}
\author{S. Srimathy}
\address{School of Mathematics, Tata Institute of Fundamental Research, Mumbai, 400005, India }
\email{srimathy@math.tifr.res.in}
\begin{abstract}
    Let $K$ be a complete discrete valued field of characteristic $p$ with residue $k$ which is not necessarily perfect.  We prove the Conjecture in \cite{cs} that a $p$-algebra over $K$ contains a totally ramified cyclic maximal subfield  if it contains a totally ramified purely inseparable maximal subfield provided $k$ satisfies some conditions on its  $p$-rank.
\end{abstract}
\maketitle
\section{Introduction}
Let $K$ denote a complete discrete valued field of chacteristic $p$ with residue $k$. The following conjecture appeared in \cite{cs}:
\begin{conjecture} (\cite{cs})\label{conj:conj}
Let $A$ be a $p$-algebra over $K$. Then $A$ contains a totally ramified  cyclic maximal subfield if and only if it contains a totally ramified purely inseparable maximal subfield.
\end{conjecture}
The "only if" part of the conjecture is completely proved in \cite{cs} and the "if" part is proved for many cases.  In  this paper, we prove the "if" part of the conjecture  for two more cases:
\begin{thm}\label{thm:conj}
    Let $A$ be a $p$-algebra  of degree $p^m, m>1$ over $K$. Suppose $A$ contains a totally ramified purely inseparable maximal subfield. Then it contains a totally ramified cyclic maximal subfield if one of the following  conditions holds:
    \begin{enumerate}
\item $rank_p(k) \geq 2m$
\item $rank_p(k) \geq  m$ and $dim_{\mathbb{F}_p}(k/\mathcal{P}(k)) \geq 1$
 \end{enumerate}

where $rank_p(k)$ denotes the $p$-rank of $k$ and $\mathcal{P}$ denotes the Artin-Schreier operator.
\end{thm}
In other words, the theorem states that the conjecture is true it $k$ admits  at least $2m$ linearly disjoint purely inseparable extensions of degree $p$  or if it admits   at least $m$ linearly disjoint purely inseparable extensions of degree $p$ and an Artin-Schreier extension.  \\
\indent The idea of the the proof involves first showing the existence of cyclic lifts of  purely inseparable extensions of exponent one over the residue .  Once we find such cyclic lifts, we construct suitable $p$-division algebras that  contain these cyclic lifts and share the  same totally ramified purely inseparable maximal subfield with $A$.  Then we use linkage results of \cite{CFM} to show the existence of totally ramified cyclic maximal subfields.
\section{Notations}
All the fields in this paper have characteristic $p$. For a field $F$, the set of non-zero elements of $F$ is denoted by $F^{\times}$.  Given a field extension $L/K$, the symbol $N_{L/K}$ denotes the norm function  of $L$ over $K$. The symbol $W_n(F)$ denotes the truncated Witt vector of length $n$ over $F$. 
 \\
 \indent Fields with discrete valuations are denoted with upper case alphabets and their residue fields are  denoted by corresponding lower case alphabets.  We denote the valuation ring, its maximal ideal and the value group of  a discrete valued field $K$ by $\mathcal{O}_K$, $\mathfrak{m}_K$ and $\Gamma_K$ respectively.   The set of natural numbers is denoted by $\mathbb{N}$.  The greatest common divisor of $m,n \in \mathbb{N}$ is denoted by $(m,n)$. 
\section{Preliminaries} \label{sec:prelim}
\subsection{Albert's theorem}
Let $F$ be a field of characteristic $p$. Let $\mathcal{P}$ denote the Artin-Schreier operator 
\begin{align*}
    \mathcal{P}: F \rightarrow F\\
                  x \mapsto x^p -x
\end{align*}
It is well known that cyclic degree $p$ extensions of $F$ upto isomorphism are in bijection with the cyclic subgroups of $F/\mathcal{P}(F)$ of order $p$ (\cite[Chapter VI, Theorem 8.3]{lang_algebra}). Let $W_n(F)$  denote the group of truncated Witt vectors of length $n$ over $F$. Given any  $\omega \in W_n(F)$, one can construct  a cyclic extension over $F$, denoted by $F_{\omega}/F$.  Conversely, given any cyclic extension  of degree $p^n$, one can associate a Witt vector of length $n$.  These are well known (\cite[Chapter III]{Jacobson:1964}, \cite{lara_thesis}).  An explicit way to construct a cyclic extension of degree $p^{n+1}$  containing any given cyclic  extension  of degree $p^n$ is due to Albert which we recall  below.
\begin{thm}(\cite[Lemma 7]{albert_cyclic}, \cite[Theorem 4.2.3]{jacobson})\label{thm:albert}
Let $F$ be a field of characteristic $p$ and let  $E/F$ be a cyclic extension of degree $p^e, e \geq 1$ with $Gal(E/F) = <\sigma>$. Then $E$ contains an element $\beta$ such that $Tr_{E/F}(\beta) =1$ and if $\beta$ is such an element, then there exists an $\alpha \in E$ such that 
\begin{align*}
    \mathcal{P}(\beta) = \sigma(\alpha) - \alpha.
\end{align*}
 Then $x^p -x - \alpha$ is irreducible in $E[x]$ and if $\gamma$ is a root of this polynomial, then $E' = E[\gamma]$ is a cyclic field of degree $p^{e+1}$ over $F$ and any such extension can be obtained this way.   
\end{thm}
The following  remark is obvious: 
\begin{remark} \normalfont \label{lem:albert}
    In the above theorem, $\alpha$ can be replaced with $c+\alpha$ for any $c\in F$ and the resulting extension is still cyclic  of degree $p^{e+1}$ over $F$.  In other words, the  extension $E[\gamma]$ where $\gamma$ is a root of $x^p -x - (c+\alpha)$ yields a cyclic extension of degree $p^{e+1}$ over $F$.
\end{remark}

\subsection{$p$-rank of a field}
Let $F$ be a field of characteristic $p$.  A set of elements $\{x_1, x_2, \cdots, x_n\}$ in $F$ is said to be \emph{$p$-independent} over $F^p$ if $[F_p(x_1, x_2, \cdots, x_n):F^p] = p^n$. This is equivalent to saying that $[F_p(x_i):F_p] =p$ and that $F_p(x_i)$ are linearly disjoint over $F_p$. If moreover, $F_p(x_1, x_2, \cdots, x_n) = F$  we say that $\{x_1, x_2, \cdots, x_n\}$ is a \emph{$p$-basis} for $F/F^p$.  We have that  $[F:F^p] = p^n$ where the integer $n$ is called the  \emph{$p$-rank (a.k.a the $p$-dimension)} of $F$, denoted by $rank_p(F)$. If $[F:F^p]$ is infinite, we set $rank_p(F)= \infty$. 

\subsection{A sufficient condition for a cyclic $p$-algebra to be a division algebra}
\indent Let $K$ be a complete discrete valued field with valuation $\mathfrak{v}$ and residue $k$. We say that an extension $L/K$ is \emph{weakly unramified} if the residue extension satisfies $[l:k]=[L:K]$ (Here we do not assume separability of $l/k$). Let $D$ be a division algebra over $K$.   The valuation on $K$ extends uniquely to a valuation on $D$ (\cite[Corollary 2.2]{wadsworth}).  Let $\overline{D}$ denote the residue algebra of $D$ and $\Gamma_D$ denote the value group of $D$. Denote the degree of $D$ over $K$ by $[D:K]$.
We now observe the following:
\begin{lemma}\label{lem:semiramified}
    Suppose $D$ contains a maximal subfield $L$ that is weakly unramified over $K$ and a maximal subfield that is totally ramified. Then, $\overline{D}= l$ where $l$ is the residue field of $L$.
\end{lemma}
\begin{proof}
    Since $L/K$ is weakly unramified, $[l:k] = [L:K] =\sqrt{[D:K]}$. Since $D$ contains a totally ramified maximal subfield, $[\Gamma_D: \Gamma_K] \geq \sqrt{[D:K]}$. By fundamental equality(\cite[Equation (2.10)]{wadsworth}, \cite[page 359]{morandi_defective}), 
\begin{align}\label{eqn:fund}
    [D:K] = [\Gamma_D: \Gamma_K] [\overline{D}: k]
    \end{align}
So we conclude that $[\overline{D}: k] \leq \sqrt{[D:K]} = [l:k]$. On the other hand, $l \subseteq \overline{D}$. Therefore, $\overline{D} = l$.
\end{proof}

Recall that any cyclic $p$-algebra of degree $p^m$ over $K$ is of the form :
\begin{align*}
K \langle x_1,\dots,x_m,y : (x_1^p,\dots,x_m^p)=(x_1,\dots,x_m)+\omega,\\ y^{p^m}=b, (yx_1y^{-1},\dots,y x_m y^{-1})=(x_1,\dots,x_m)+(1,0,\dots,0)\rangle
\end{align*}
for some $\omega \in W_m(K)$, $b \in K^{\times}$ and  the `$+$' above  denotes addition rule of the Witt vectors. We denote it by the symbol $[\omega, b)_K$.\\
\indent The following lemma which a simple variant of \cite[Lemma 3.3]{cs}, gives a sufficient condition for a $p$-algebra over $K$ to be a division algebra. The proof is similar to the proof of \cite[Lemma 3.3]{cs}, so we skip it.
\begin{lemma}\label{lem:division}
    Let $b \in K^{\times}$ with $(\mathfrak{v}(b),p) = 1$. Let $\omega \in W_m(K)$ be such that the corresponding cyclic extension $K_{\omega}/K$ is weakly unramified of degree $p^m$.  Then  $[\omega, b)_K$ is a division algebra over $K$.
\end{lemma}

\section{Cyclic extensions with inseparable residue  of exponent one}
The goal of this section is to prove the following.
\begin{thm} \label{thm:cyclic}
    Let $K$ be a complete discrete valued field of characteristic $p$ with  valuation $\mathfrak{v}$ and  residue $k$. Suppose $\{\overline{a_1}, \overline{a_2} \cdots,  \overline{a_m}\} \subset k^{\times}$ be elements  that are $p$-independent over $k^{p}$. Then there exists a cyclic extension $L/K$ of degree $p^m$ whose residue  is $l = k(\sqrt[p]{\overline{a_1}}, \sqrt[p]{\overline{a_2}}, \cdots, \sqrt[p]{\overline{a_m})}$.
\end{thm}
\begin{proof}
    Let $t$ be a uniformizer of $K$ and let $\{a_1, a_2, \cdots, a_m\} \subset \mathcal{O}_K^{\times}$ be arbitrary lifts of  $\{\overline{a_1}, \overline{a_2} \cdots,  \overline{a_m}\} \subset k^{\times}$ .  We will use induction on $m$ to build $L_1, L_2, \cdots, L_m = L$ such that  $L_i/K$ is cyclic and the residue field of  $L_i$ is  $l_i = k(\sqrt[p]{a_1}, \sqrt[p]{a_2}, \cdots, \sqrt[p]{a_i)}$. Let $m=1$. Consider the Artin-Schreier extension $L_1/K$ given by
    \begin{align*}
        x_1^p - x_1 = \frac{a_1}{t^p}
    \end{align*}
    Let $y_1 = tx_1$.  Then $y_1$ satisfies
    \begin{align*}
        y_1^p - (t^{p-1})y_1 = a_1
    \end{align*}
    Clearly $y_1 \in \mathcal{O}_{L_1}^{\times}$ and its residue is 
    \begin{align*}
        \overline{y_1} = \sqrt[p]{a_1}
    \end{align*}
    Therefore the residue field $l_1$ of $L_1$ is given by
    \begin{align*}
        l_1 = k(\sqrt[p]{a_1})
    \end{align*}
    Now assume by induction hypothesis that we have constructed cyclic $L_{i-1}/K$ with residue field $l_{i-1} = k(\sqrt[p]{a_1}, \sqrt[p]{a_2}, \cdots, \sqrt[p]{a_{i-1})}$. By Theorem \ref{thm:albert} and Remark  \ref{lem:albert}, there exists $\alpha_{i-1} \in L_{i-1}$ such that the Artin-Schreier extension $L_i/L_{i-1}$ given by 
    \begin{align}\label{eqn:temp1}
        x_{i}^p - x_{i} = \frac{a_i}{t^{p^{n_i}}} + \alpha_{i-1}
    \end{align}
    is such that $L_i/K$ is cyclic. Choose $n_i \in \mathbb{Z}, n_i \geq 1$ such that $\mathfrak{v}(t^{p^{n_i}})> - \mathfrak{v}(\alpha_{i-1})$ and let $y_i = t^{p^{n_i-1}}x_i$. Then from (\ref{eqn:temp1}), we get
    \begin{align*}
        y_i^p - (t^{p^{n_i}-p^{n_i-1}})y_i = a_i + \beta_{i-1}
    \end{align*}
    where $\beta_{i-1} = t^{p^{n_i}}\alpha_{i-1} \in \mathfrak{m}_{L_{i-1}}$ by the choice of $n_i$. From the above equation, we see that $y_i \in \mathcal{O}_{L_i}^{\times}$ and its residue is given by 
    \begin{align*}
        \overline{y_i} = \sqrt[p]{a_i}
    \end{align*}
    Therefore the residue field $l_i$ of $L_i$ is given by
    \begin{align*}
        l_i = l_{i-1}(\sqrt[p]{a_i}) = k(\sqrt[p]{a_1}, \sqrt[p]{a_2}, \cdots, \sqrt[p]{a_i})
    \end{align*}
    as claimed.
\end{proof}

\section{Proof of Theorem \ref{thm:conj}}
\begin{lemma}\label{lem:disjoint}
     Let $m\in \mathbb{N}$ and let $k$ satisfy either of the following conditions: 
        \begin{enumerate}
\item $rank_p(k) \geq 2m$
\item $rank_p(k) \geq  m$ and $dim_{\mathbb{F}_p}(k/\mathcal{P}(k)) \geq 1$
 \end{enumerate}
 Then there exists weakly unramified cyclic extensions $L_1/K$ and $L_2/K$ of degree $p^m$ whose residue satisfy  $l_1 \cap l_2 = k$.
\end{lemma}
\begin{proof}
   \begin{enumerate}
       
       \item When $rank_p(k) \geq 2m$: Let $\{\overline{a_1}, \overline{a_2} \cdots,  \overline{a_{2m}}\} \subset k^{\times}$ be elements  that are $p$-independent over $k^{p}$.  Let $l_1=k(\sqrt[p]{\overline{a_1}}, \sqrt[p]{\overline{a_2}}, \cdots, \sqrt[p]{\overline{a_m})}$ and $l_2 = k(\sqrt[p]{\overline{a_{m+1}}}, \sqrt[p]{\overline{a_{m+2}}}, \cdots, \sqrt[p]{\overline{a_{2m}})}$. Then clearly $l_1\cap l_2 = k$. Let $L_i/K, i=1,2$ be the cyclic extensions of degree $p^m$, whose residue is $l_i/k$ as constructed in Theorem \ref{thm:cyclic}. 
       \item When $rank_p(k) \geq  m$ and $dim_{\mathbb{F}_p}(k/\mathcal{P}(k)) \geq 1$: Since $rank_p(k) \geq  m$, there exists $\{\overline{a_1}, \overline{a_2} \cdots,  \overline{a_{m}}\} \subset k^{\times}$  that are $p$-independent over $k^{p}$.  Let $l_1=k(\sqrt[p]{\overline{a_1}}, \sqrt[p]{\overline{a_2}}, \cdots, \sqrt[p]{\overline{a_m})}$. Also since $dim_{\mathbb{F}_p}(k/\mathcal{P}(k)) \geq 1$, there exists an Artin-Schreier extension of $k$. By Theorem \ref{thm:albert}, we can extend this to a cyclic extension $l_2/k$ of degree $p^m$.  Clearly, $l_1\cap l_2 =k$. Let   $L_1/K$ be the cyclic extension of degree $p^m$, whose residue is $l_1/k$ as constructed in Theorem \ref{thm:cyclic} and let $L_2/K$ be the inertial lift of $l_2/k$ which is again cyclic of degree $p^m$. 
       \end{enumerate}
       In both the  cases,  $L_1$ and $L_2$ clearly satisfy the lemma.
\end{proof}

We are now ready to prove Theorem \ref{thm:conj}.

\begin{proof}[Proof of Theorem \ref{thm:conj}:]
   Let   $A$ be a $p$-algebra of degree $p^m$ over $K$ containing a totally ramified purely inseparable maximal subfield. Then by \cite[Lemma 5.1, Remark 5.2]{cs}, $A \simeq [\omega, b)$ for some $\omega \in W_m(F)$, $b \in K^{\times}$ with $(\mathfrak{v}(b),p)=1$ and $K(\sqrt[p^m]{b}) \subset A$ is totally ramified purely inseparable over $K$. Let $L_i/K, i=1,2$ be the degree $p^m$ cyclic extension constructed in Lemma \ref{lem:disjoint}. Let $\omega_i \in W_m(K)$ be Witt vectors corresponding to $L_i$. Since $L_1$ and $L_2$ are weakly unramified, by Lemma \ref{lem:division}, $D_i = [\omega_i, b)_K$, $i=1,2$ are   division algebras over $K$. Moreover, $\overline{D_i} = l_i$ by Lemma \ref{lem:semiramified}. Now $A$, $D_1$ and $D_2$ share  the same purely inseparable subfield $K(\sqrt[p^m]{b})$ and therefore by \cite[Theorem 4.7]{CFM}, share a cyclic maximal subfield $L$. Now the residue $l$ of $L$ satisfies
  \begin{align*}
  l\subseteq \overline{D_1} \cap \overline{D_2} = l_1 \cap l_2 = k
\end{align*}  
By the fundamental equality \cite[Chapter II, \S2, Corollary 1]{serre_local}, $L/K$ is totally ramified as required.
    \end{proof}
\section*{Acknowledgements}
The author acknowledges the support of the DAE, Government of India, under Project Identification No. RTI4001. 
\nocite*{}
\bibliographystyle{alpha}
\bibliography{reference}
\end{document}